\documentclass[12pt]{amsart}
\usepackage{amsmath,amssymb,amsbsy,amsfonts,latexsym,amsopn,amstext,
                                               amsxtra,euscript,amscd,bm}
\usepackage{url}
\usepackage[colorlinks,linkcolor=blue,anchorcolor=blue,citecolor=blue]{hyperref}
\usepackage{color}
\usepackage{float}

\hypersetup{breaklinks=true}

\begin{document}

\newtheorem{theorem}{Theorem}
\newtheorem{lemma}[theorem]{Lemma}
\newtheorem{claim}[theorem]{Claim}
\newtheorem{cor}[theorem]{Corollary}
\newtheorem{prop}[theorem]{Proposition}
\newtheorem{definition}{Definition}
\newtheorem{question}[theorem]{Question}
\newcommand{\hh}{{{\mathrm h}}}

\numberwithin{equation}{section}
\numberwithin{theorem}{section}
\numberwithin{table}{section}

\def\sssum{\mathop{\sum\!\sum\!\sum}}
\def\ssum{\mathop{\sum\ldots \sum}}
\def\iint{\mathop{\int\ldots \int}}

\def\squareforqed{\hbox{\rlap{$\sqcap$}$\sqcup$}}
\def\qed{\ifmmode\squareforqed\else{\unskip\nobreak\hfil
\penalty50\hskip1em\null\nobreak\hfil\squareforqed
\parfillskip=0pt\finalhyphendemerits=0\endgraf}\fi}

\newfont{\teneufm}{eufm10}
\newfont{\seveneufm}{eufm7}
\newfont{\fiveeufm}{eufm5}
%
%
\newfam\eufmfam
     \textfont\eufmfam=\teneufm
\scriptfont\eufmfam=\seveneufm
     \scriptscriptfont\eufmfam=\fiveeufm
%
%
\def\frak#1{{\fam\eufmfam\relax#1}}

\newcommand{\bflambda}{{\boldsymbol{\lambda}}}
\newcommand{\bfmu}{{\boldsymbol{\mu}}}
\newcommand{\bfxi}{{\boldsymbol{\xi}}}
\newcommand{\bfrho}{{\boldsymbol{\rho}}}

\def\fK{\mathfrak K}
\def\fT{\mathfrak{T}}

\def\fA{{\mathfrak A}}
\def\fB{{\mathfrak B}}
\def\fC{{\mathfrak C}}

\def\eqref#1{(\ref{#1})}

\def\vec#1{\mathbf{#1}}

\def\squareforqed{\hbox{\rlap{$\sqcap$}$\sqcup$}}
\def\qed{\ifmmode\squareforqed\else{\unskip\nobreak\hfil
\penalty50\hskip1em\null\nobreak\hfil\squareforqed
\parfillskip=0pt\finalhyphendemerits=0\endgraf}\fi}

\def\cA{{\mathcal A}}
\def\cB{{\mathcal B}}
\def\cC{{\mathcal C}}
\def\cD{{\mathcal D}}
\def\cE{{\mathcal E}}
\def\cF{{\mathcal F}}
\def\cG{{\mathcal G}}
\def\cH{{\mathcal H}}
\def\cI{{\mathcal I}}
\def\cJ{{\mathcal J}}
\def\cK{{\mathcal K}}
\def\cL{{\mathcal L}}
\def\cM{{\mathcal M}}
\def\cN{{\mathcal N}}
\def\cO{{\mathcal O}}
\def\cP{{\mathcal P}}
\def\cQ{{\mathcal Q}}
\def\cR{{\mathcal R}}
\def\cS{{\mathcal S}}
\def\cT{{\mathcal T}}
\def\cU{{\mathcal U}}
\def\cV{{\mathcal V}}
\def\cW{{\mathcal W}}
\def\cX{{\mathcal X}}
\def\cY{{\mathcal Y}}
\def\cZ{{\mathcal Z}}
\newcommand{\rmod}[1]{\: \mbox{mod} \: #1}

\def\vr{\mathbf r}

\def\e{{\mathbf{\,e}}}
\def\ep{{\mathbf{\,e}}_p}
\def\em{{\mathbf{\,e}}_m}

\def\Tr{{\mathrm{Tr}}}
\def\Nm{{\mathrm{Nm}}}

 \def\SS{{\mathbf{S}}}

\def\lcm{{\mathrm{lcm}}}

\def\({\left(}
\def\){\right)}
\def\fl#1{\left\lfloor#1\right\rfloor}
\def\rf#1{\left\lceil#1\right\rceil}

\def\mand{\qquad \mbox{and} \qquad}

         \newcommand{\comm}[1]{\marginpar{\vskip-\baselineskip
         \raggedright\footnotesize
\itshape\hrule\smallskip#1\par\smallskip\hrule}}




\hyphenation{re-pub-lished}

\parskip 4pt plus 2pt minus 2pt

\mathsurround=1pt

\def\bfdefault{b}
\overfullrule=5pt

\def \F{{\mathbb F}}
\def \K{{\mathbb K}}
\def \Z{{\mathbb Z}}
\def \Q{{\mathbb Q}}
\def \R{{\mathbb R}}
\def \C{{\\mathbb C}}
\def\Fp{\F_p}
\def \fp{\Fp^*}

\title[Bilinear Exponential Sums with 
Reciprocals]{On Bilinear Exponential and Character  Sums with 
Reciprocals of Polynomials}

 \author[I. E. Shparlinski] {Igor E. Shparlinski}
 \thanks{This work was  supported in part by ARC Grant~DP140100118.}

\address{Department of Pure Mathematics, University of New South Wales,
Sydney, NSW 2052, Australia}
\email{igor.shparlinski@unsw.edu.au}


\begin{abstract} We  give nontrivial bounds for the  bilinear 
sums
$$
\sum_{u =  1}^{U} \sum_{v=1}^V \alpha_u \beta_v \ep(u/f(v))
$$
where $\ep(z)$ is a nontrivial additive character of the prime finite 
field $\F_p$ of $p$ elements, 
with integers $U$, $V$, a polynomial $f\in \F_p[X] $
and some complex weights $\{\alpha_u\}$, $\{\beta_v\}$.
In particular, for $f(X)=aX+b$ we obtain new bounds
of bilinear sums with Kloosterman fractions. 
We also obtain new bounds for similar sums with 
multiplicative characters of $\F_p$. 
 \end{abstract}


\maketitle

\section{Introduction} 
\subsection{Background and motivation}
\label{sec:intro}

 Let $\F_p$ denote the finite 
 field of $p$ elements, where $p$ is   a sufficiently large 
prime. Assume that we are given two integers $1 \le U,V < p$, a {\it convex\/} set 
$$
\fC \subseteq [1,U]\times [1,V], 
$$
a polynomial $f\in \F_p[X]$ and two sequences 
of complex ``weights'' 
$$\cA = \{\alpha_u\} \mand \cB =\{\beta_v\} \quad \text{with} \quad 
|\alpha_u|, |\beta_v| \le 1.
$$
We then consider the bilinear exponential and character sums
\begin{equation*}
\begin{split}
&S_f(\cA, \cB; \fC)= 
\sum_{(u,v)\in \fC} \alpha_u \beta_v \e_p(v/f(u)),
\\ 
&T_f(\cA, \cB; \fC)= 
\sum_{(u,v)\in \fC} \alpha_u \beta_v \chi(v+f(u)),
\end{split}
\end{equation*}
where $\e_p(z) = \exp(2 \pi i z/p)$ and $\chi$ is a fixed 
nonprincipal character of $\F_p^*$, we refer to~\cite{IwKow} 
for a background on exponential sums and 
multiplicative characters. To simplify the notation 
we always assume that the zeros of $f$  in the set $\{1, \ldots, U\}$
are excluded from the summation in $S_f(\cA, \cB; \fC)$ 
(without adding any extra condition
on the summation ranges). Note that  with the weights
$\widetilde  \alpha_u  =  \alpha_u\chi(f(u))$ 
the sums
$T_f(\cA, \cB; \fC)$ can be written in the same shape as 
$S_f(\cA, \cB; \fC)$, that is,
$$
T_f(\cA, \cB; \fC)= 
\sum_{(u,v)\in \fC}\widetilde  \alpha_u \beta_v \chi(v/f(u)+1).
$$ 

Using the well known general bound of bilinear sums, 
for $\Pi = [1,U]\times [1,V]$, we have 
$|S_f(\cA, \cB; \Pi)| \le  \sqrt{UVp}$,
see, for example,~\cite[Equation~(1.4)]{BouGar1}, 
and a similar bound for $T_f(\cA, \cB; \Pi)$, which we use as the benchmarks of 
our progress. Using the same techniques as in the 
proof of Theorems~\ref{thm:Sabd mix} and~\ref{thm:K mix 1} below,  it is easy to extend these bounds
to arbitrary convex domains $\fC$ with just a logarithmic loss:
 \begin{equation}
\label{eq:trivial S1}
S_f(\cA, \cB; \fC)  = O\(\sqrt{UVp} \log p\)
 \end{equation}
and 
 \begin{equation}
\label{eq:trivial T}
T_f(\cA, \cB; \fC)  = O\(\sqrt{UVp} \log p\)
\end{equation}
(in fact Theorem~\ref{thm:Sabd mix} with $k=1$ essentially 
gives~\eqref{eq:trivial S1} as well).

In the special case of  the weights $\beta_v =1$, $v \in [1,V]$, 
we simply write $S_f(\cA;\fC)$ and $T_f(\cA;\fC)$ for the 
corresponding sums. 

Furthermore, for linear polynomials $f(X) = aX+b$ with $a,b\in \F_p$ we write 
$$
K_{a,b}(\cA, \cB; \fC) 
=\sum_{(u,v)\in \fC} \alpha_u \beta_v \e_p\(v/(au+b)\),
$$
and
$$
K_{a,b}(\cA,\fC)=\sum_{(u,v)\in \fC} \alpha_u \e_p\(v/(au+b)\), 
$$
for the  sums $S_f(\cA, \cB; \fC)$ and $S_f(\cA;\fC)$, respectively. 
These sums have a natural interpretation 
as bilinear Kloosterman sums. We also note that similar trilinear 
sums (with additional 
averaging over the modulus $p$) have recently been considered
by  Bettin \& Chandee~\cite{BetChan}; we also recall the works
of Bourgain~\cite{Bour} and Bourgain \& Garaev~\cite{BouGar2}
where different types of  bilinear Kloosterman sums
are studied.

With these notations, one of the results of~\cite{Shp2} can be written 
as the uniform over $a \in \F_p^*$ and $b\in \F_p$  bound 
\begin{equation}
\label{eq:K pure}
|K_{a,b}(\cA;\fC)| \le (\sqrt{Up}+V)p^{o(1)}, 
\end{equation}
as $p\to \infty$. 
Furthermore, for $b = 0$, a stronger  bound 
$$
|K_{a,0}(\cA;\fC)| \le (U+V)p^{o(1)}, 
$$
has been given in~\cite{Shp1}, which is essentially optimal
when $U$ and $V$ are of the same order. 

Note that in~\cite{Shp1,Shp2} only the case of $\alpha_u =1$ 
is considered, but the proofs work without any changes  
for aribtary weights $\cA$. 

\subsection{Our results}

Here we obtain a series of   results, similar to the bound~\eqref{eq:K pure} 
for the sums
$S_f(\cA;\fC)$, $T_f(\cA;\fC)$ and $K_f(\cA, \cB; \fC)$. 
These bounds are nontrivial starting from rather small values 
of $U$ and $V$, in particular in the case when the product $UV$ 
is much smaller than $p$, which is necessary for the 
bounds~\eqref{eq:trivial S1} and~\eqref{eq:trivial T} to be nontrivial. 

Our results are closely related to various questions
about congruences with reciprocals and 
multiplicative congruences with polynomials of the types 
considered in~\cite{BouGar2} and~\cite{Chang,CCGHSZ,CillGar,CGOS,CSZ,Kerr}, 
respectively. We also use one of the results from~\cite{ChanShp}
which  we extend to more generic settings; we hope this may find 
further applications. 

\subsection{Notation}

Throughout the paper, any implied constants in the symbols $O$,
$\ll$  and $\gg$  may depend on the real  parameter $\varepsilon> 0$ and
the integer parameters $d, \nu\ge 1$. We recall that the
notations $A = O(B)$, $A \ll B$ and  $A \gg B$ are all equivalent to the
statement that the inequality $|A| \le c B$ holds with some
constant $c> 0$.

When we say that a polynomial $f(X) \in \F_p[X]$ 
is of degree $d\ge 1$ we always mean the exact degree, 
that is, the leading term of $f$ is $aX^d$ with $a \in \F_p^*$. 

The elements of the field $\F_p$ 
are assumed to be represented by the set $\{0, \ldots, p-1\}$.
In particular,  we often treat elements of $\F$ as integers
or  residue classes modulo $p$. 

\section{Main results}

\subsection{Bilinear exponential sums with polynomials}

We start with extending~\eqref{eq:K pure} to arbitrary 
polynomials. 
 
\begin{theorem}
\label{thm:S pure} Uniformly over  polynomials $f(X) \in \F_p[X]$ 
of degree $d\ge 1$,  
we have 
$$
|S_f(\cA;\fC)| \le p^{d/(d+1)+o(1)}  U^{d/2}+ V p^{o(1)}. 
$$
\end{theorem}

Clearly, for  the bound of Theorem~\ref{thm:S pure}  to  be nontrivial 
it is necessary to have $U \ge  p^{\varepsilon}$ for some fixed $\varepsilon>0$. 
However this is not sufficient. For example,  and 
for $d=1$ and $d=2$  we also need 
$U  V^2 \ge p^{1 + \varepsilon}$ and $V \ge p^{2/3 + \varepsilon}$, respectively. 
Furthermore, for  $d \ge 3$,  the bound of Theorem~\ref{thm:S pure} is nontrivial
only for rather small values of $U$. Namely, for  $d \ge 3$  we also need 
$$
U  \le  p^{-2d/(d+1)(d-2)-\varepsilon} V^{2/(d-2)}. 
$$
This however can be used to derive a bound which is nontrivial for 
any $U$ and $V$ with 
$$
U  \ge  p^{\varepsilon} \mand  
V \ge  p^{d/(d+1) + \varepsilon} . 
$$

\begin{cor}
\label{cor:S pure slice} 
Uniformly over  polynomials $f(X) \in \F_p[X]$ 
of degree $d\ge 1$,  
we have 
$$
|S_f(\cA;\fC)| \le
\left\{
\begin{array}{ll}
V p^{o(1)} & \text{ if } U  \le  V^{2/d}  p^{-2/(d+1)},\\
 UV^{1-2/d} p^{2/(d+1)+o(1)} & \text{ otherwise}.
\end{array}
\right. 
$$
\end{cor}

Unfortunately this method of proof of Theorem~\ref{thm:S pure}
does not seem to apply to the general sums $S_f(\cA, \cB; \fC)$
even in the most interesting case when $\fC = [1,U]\times [1,V]$. 
However for powers of linear functions, that is 
for polynomials of the form $f(X) = (aX+b)^d\in \F_p[X]$, 
we are abble to estimate these sums, which  we denote 
by   $S_{a,b,d}(\cA, \cB; \fC)$ in this case.

Using some recent results of Bourgain \& Garaev~\cite{BouGar2}
we derive the following estimate:

\begin{theorem}
\label{thm:Sabd mix} Uniformly over   $(a,b) \in \F_p^* \times \F_p$ and any fixed 
integer $k \ge 1$, 
we have 
$$
|S_{a,b,d}(\cA, \cB; \fC)| \le
V^{1-1/2k}\(U  + U^{k/(k+1)} p^{1/2k}\)  p^{o(1)}. 
$$
\end{theorem}

It is easy to see that Theorem~\ref{thm:Sabd mix}, used with a 
sufficiently large $k$,  is
nontrivial provided that 
$$
U^2V \ge  p^{1 +\varepsilon} \mand  
V \ge  p^{\varepsilon} , 
$$
 for some fixed $\varepsilon>0$. 

\subsection{Bilinear Kloosterman sums}

The bound of Theorem~\ref{thm:Sabd mix} with $d=1$ also applies to the sums 
$K_{a,b}(\cA, \cB; \fC)$. However, in this case we can obtain 
more precise results. For example, using  
some result of Cilleruelo \& Garaev~\cite{CillGar}
in the argument on the proof of Theorem~\ref{thm:Sabd mix}
we obtain:

\begin{theorem}
\label{thm:K mix 1} Uniformly over $(a,b) \in \F_p^* \times \F_p$, we have 
$$
|K_{a,b}(\cA, \cB; \fC)|
\le  
 V^{3/4}\(U^{7/8}p^{1/8}+U^{1/2}p^{1/4}\) p^{o(1)}. 
 $$
\end{theorem}

It is easy to see that Theorem~\ref{thm:K mix 1} is
nontrivial provided that 
$$
UV^2 \ge  p^{1 +\varepsilon} \mand  
U^2V \ge p^{1+\varepsilon} , 
$$
 for some fixed $\varepsilon>0$.

Finally, using a different approach we also obtain:

\begin{theorem}
\label{thm:K mix 2} Uniformly over   $(a,b) \in \F_p^* \times \F_p$ 
we have 
$$
|K_{a,b}(\cA, \cB; \fC)| \le
UV^{1/2} p^{1/3+o(1)}+  U^{1/2}V p^{o(1)}. 
$$
\end{theorem}

Since the proof of Theorem~\ref{thm:K mix 2} uses Theorem~\ref{thm:S pure}
with $d=2$,  it is natural that Theorem~\ref{thm:K mix 2} is
nontrivial for 
$$
U  \ge  p^{\varepsilon} \mand  
V \ge  p^{2/3 + \varepsilon},  
$$
with  some fixed $\varepsilon>0$.

\subsection{Bilinear character sums with polynomials}

Clearly, if $V \ge p^{1/4+\varepsilon}$ for some 
fixed $\varepsilon > 0$ then  for the sums  $T_f(\cA;\fC)$ the  Burgess bound,
(see~\cite[Theorem~12.6]{IwKow}), applied to the sum over $v$
implies a nontrivial estimate of the shape 
$|T_f(\cA;\fC)| \le UV p^{-\delta}$, with $\delta > 0$ that depends only 
on $\varepsilon$.

Here we use some ideas and results from 
the proof of~\cite[Theorem~1.2]{ChanShp}
to estimate  the sums $T_f(\cA;\fC)$ for smaller values of $V$. 
 
\begin{theorem}
\label{thm:T pure}  For any  $\varepsilon > 0$ there exists some $\delta>0$ such that 
uniformly over  polynomials $f(X) \in \F_p[X]$ 
of degree $d\ge 3$,  for
$$
U \ge p^{1/d +\varepsilon} \mand V \ge p^{1/4-\delta}
$$
we have 
$$
|T_f(\cA;\fC)| \le UV p^{-\delta}. 
$$
\end{theorem}

We remark that in case of polynomials $f$ of degree $d=1$
much stronger results are given by  Karatsuba~\cite{Kar}.
The case of $d =2$ can also be easily included via the use 
of classical bounds of Gaussian sums instead of 
Lemma~\ref{lem:Wooley}, see below.  In fact for $U \ge p^{1/2 +\varepsilon}$ 
more standard approaches work as well, see the discussion in 
Section~\ref{sec:comm}.  




\section{General Results}

\subsection{Small residues of multiples}
\label{sec:small res}

For an integer $a$ we use $\rho(a)$ to denote the smallest by
absolute value residue of $a$ modulo $p$, that is
$$
\rho(a)= \min_{k\in \Z} |a - kp|.
$$
We need the following simple statement which follows from 
the Dirichlet pigeon-hole principle, 
see~\cite[Lemma~3.2]{CSZ} or~\cite[Theorem~2]{GG}.

\begin{lemma}
\label{lem:Reduction}
For any real numbers $T_0,\ldots, T_d$, with
$$
p> T_0,\ldots, T_d \ge 1 \mand   T_0\cdots T_d > p^{d}, 
$$
and any
integers $a_0, \ldots, a_d$ there exists an integer $t$ with $\gcd(t,p) =1$
and such that
$$
\rho(a_i t) \ll T_i, \qquad i =0, \ldots, d.
$$
\end{lemma}

\subsection{Moments of short character sums}

We recall the classical result of Davenport \&
Erd{\H o}s~\cite{DavErd}, which 
 follows from the Weil bound of multiplicative character sums,
see~\cite[Theorem~11.23]{IwKow}.

\begin{lemma}
\label{lem:DavErd} 
For a fixed integer $\nu \ge 1$ and a positive 
integer $K < p$, we have
$$
 \sum_{\lambda \in \F_p}  
 \left|\sum_{k =1}^K
\chi\(\lambda + r\)\right|^{2\nu} \ll K^{2\nu} p^{1/2} + K^\nu p.
$$
\end{lemma}

\subsection{Congruences with uniformly distributed sequences}
\label{sec:u-distr cong}

We say that a sequence $\cR = \{r_u\}_{u=1}^U$ of  $U$ elements of $\F_p$  is  $\eta$-uniformly distributed 
modulo $p$ if uniformly over  $b \in \F_p$ and for any positive integer $Z <p$
\begin{equation}
\label{eq:eta distr}
\begin{split}
\#\{(u, z) \in [1,N] \times [1,Z]~:~ r_u  \equiv b &+ z \pmod p\}\\
&= \frac{U  Z}{p} + O(\# \cR p^{-\eta}). 
\end{split}
\end{equation}

We now establish a more general form of~\cite[Lemma~2.9]{ChanShp}, 
which in turn is based on some ideas and results 
of Bourgain,  Konyagin and  Shparlinski~\cite{BKS} and Shao~\cite{Shao}.

\begin{lemma}
\label{lem:N bound}  For any $\eta>0$ there exist some $\kappa>0$ such that 
for positive integers $L < V < p^{1/2-\eta}$  and  any $\eta$-uniformly distributed modulo $p$  sequence 
 $\cR = \{r_u\}_{u=1}^U$ of  $U$ elements of $\F_p$,   for 
 \begin{equation*}
\begin{split}
N  =\# \Bigl\{(\ell_1,\ell_2,u_1,u_2,v_1,v_2,)&\in \cL^2 \times [1, U]^2 \times [1, V]^2~:\\
&~
\frac{v_1+r_{u_1}}{\ell_1} \equiv \frac{v_2+r_{u_2}}{\ell_2} \pmod p\Bigr\},
\end{split}
\end{equation*}
where $\cL$  is the set of primes of the interval $[L, 2L]$, 
we have
$$
N\ll  L U^2 Vp^{-\kappa} .
$$
\end{lemma}

\begin{proof}  We say that a sequence $t_1, \ldots, t_n \subseteq \F_p$ is $V$-spaced
if no  integers $1\le i < j \le n $ with $i\ne j$ and an  integer $v$ with $|v|\le V$
satisfy the equality $t_i+v=t_j$.

Let $\cS_1$ be the set of indices of  the largest $V$-separated subsequence 
of $\cU_0 = \{1, \ldots, U\}$. 

Since $\cR$ is $\eta$-uniformly distributed modulo $p$, we conclude that 
for some constant $C$, 
depending only on the implied constant in~\eqref{eq:eta distr}, each interval 
$[b,b+Z] \in [1,p-1]$ of length $Z \ge C p^{1-\eta} \ge V$ contains an element 
of $\cR$.  So, covering $[1,p]$ 
by non-overlapping intervals of length $\rf{C p^{1-\eta}}$ and choosing 
an element of the sequence $\cR$ in every second interval (except possibly the
last one) we see that $\# \cS_1 \gg p^{\eta}$.

We now set $\cU_1= \cU_0\setminus \cS_1$ to be set of indices  of remaining 
elements of $\cR_0$ and proceed inductively,
defining  $\cS_{k+1}$ as the largest set of indices of a 
$V$-separated subsequence of the sequence $\{r_u\}_{u\in \cU_k}$, where
$$
\cU_{k} = \cU_{k-1} \setminus  \cS_k, 
\qquad k =1, 2, \ldots.
$$
We terminate, when $\cS_{k+1} \le p^{\eta/2}$ and
then use $K$ to denote this value of $k$ and also set $\cS_0 = \cU_k$.

Hence there is a partition 
$$
\cU_0 = \bigcup_{k=0}^{K} \cS_k 
$$
into $K \le  U p^{-\eta/2}$  disjoined subsets,
where  $\{r_u\}_{u \in \cS_k}$ is a $V$-separated 
sequence with  $\# \cS_k \ge p^{\eta/2}$, 
$k=1, \ldots, K$.

We claim that 
\begin{equation}
\label{eq:S0 small}
\# \cS_{0} \ll U p^{-\eta/2}.
\end{equation}
Indeed, let $N$ be the smallest number of 
intervals of length $V$ that covers the set 
 $\{r_u\}_{u \in \cS_{0}}$. 
Clearly $\{r_u\}_{u \in \cS_{0}}$ contains a $V$-separated set 
of size $\rf{N/2}$. Hence $N \ll p^{\eta/2}$.
Since $\{r_u\}_{u \in \cS_{0}}$ is a subsequence 
of the sequence $\cR$, we
see that each of such intervals in this covering 
contains at most 
$UV/p + O(U p^{-\eta}) \ll U p^{-\eta}$ elements of 
$\cR$ and therefore of  $\{r_u\}_{u \in \cS_{0}}$. 
Therefore $\# \cS_{0} \ll N U  p^{-\eta}$ and~\eqref{eq:S0 small}
follows. 
The rest of the proof is identical to that of~\cite[Lemma~2.9]{ChanShp}.
\end{proof}

\subsection{Character sums with  uniformly distributed sequences}
\label{sec:u-distr char}

We now use the argument of the proof~\cite[Theorem~1.1]{ChanShp} to estimate certainly 
double character sums with $\eta$-uniformly distributed sequences 
modulo $p$ (as defined in Section~\ref{sec:u-distr cong}.

We use the same notatio a convex set 
$\fC \subseteq [1,U]\times [1,V]$ and weights 
$\cA $  as in Section~\ref{sec:intro}.

\begin{lemma}
\label{lem:CharSum bound}  For any $\eta>0$ there exist some $\delta>0$ such that 
for $V \ge p^{1/4-\delta}$  any $\eta$-uniformly distributed
modulo $p$   sequence 
 $\cR = \{r_u\}_{u=1}^U$ of  $U$ elements of $\F_p$, we have
$$
 \sum_{(u,v)\in \fC}  \alpha_u  
   \chi(v+r_u) \ll UVp^{-\delta}.
 $$
\end{lemma}

\begin{proof}  
We note that 
since $\fC$ is convex,  for each $u=1, \ldots, U$, there  are
 integers $V \ge Y_u \ge X_u \ge 0$ such that 
$$
\sum_{(u,v)\in \fC}  \alpha_u  
   \chi(v+r_u)=
\sum_{u =  1}^{U} \alpha_u  \sum_{v=X_u+1}^{Y_u}  \chi(v+r_u).
$$
Clearly we can assume that $V< p^{1/3}$ as otherwise the Burgess bound 
(see~\cite[Theorem~12.6]{IwKow}) implies the desired result. 
Hence, assuming that $\eta$ is small enough we see that  
the conditions of Lemma~\ref{lem:N bound}  are satisfied 
for the sequence $\cR$. 

For $\kappa$ that corresponds the above value of $\eta$ in 
Lemma~\ref{lem:N bound}, we set 
$$
\gamma = \kappa/5.
$$ 

Let $K =   \rf{p^{\gamma}}$ and $L = Vp^{-2\gamma}$. 
Also, as in Lemma~\ref{lem:N bound} we use $\cL$ 
to denote the set of primes of the interval $[L, 2L]$. 

Thus for any integer $w$ we have 
\begin{equation*}
\begin{split}
\sum_{(u,v)\in \fC}  \alpha_u  
   \chi(v+r_u) &=
\sum_{u =  1}^{U} \alpha_u \(\sum_{v=X_u+1}^{Y_u}  \chi(v+r_u + w) + O(|w|)\)\\
& = \sum_{(u,v)\in \fC} \alpha_u  \chi(v+r_u + w) +O(U|w|).
\end{split}
\end{equation*}
Therefore 
\begin{equation}
\begin{split}
\label{eq:Sigma}
\sum_{(u,v)\in \fC}  \alpha_u  
   \chi(v+r_u)& = \frac{1}{K \#\cL}
\fT  + O(K L U) \\
& \ll  \frac{1}{KL} \fT \log p+ O(VUp^{-\gamma}),
\end{split}
 \end{equation}
where
$$
\fT   = \sum_{(u,v)\in \fC} \sum_{\ell \in \cL} \sum_{k =1}^K
 \chi(v+r_u +  k \ell) .
$$
Hence 
\begin{equation*}
\begin{split}
\fT  & \le   \sum_{(u,v)\in \fC}  
\sum_{\ell \in \cL} \left|\sum_{k =1}^K
  \chi(v+r_u +  k \ell)\right|\\
& \le   \sum_{u=1}^U \sum_{v=1}^V  
\sum_{\ell \in \cL}\left| \sum_{k =1}^K
 \chi(v+r_u +  k \ell)\right|\\
& =  \sum_{u=1}^U \sum_{v=1}^V  
\sum_{\ell \in \cL} \left|\sum_{k =1}^K
  \chi\(\frac{v+r_u}{\ell}+  k\)\right|. 
\end{split}
\end{equation*}
Collecting together the triples $(u,v,\ell)$ 
with the same value $(v+r_u)/\ell$,  we obtain 
$$
\fT   \le  \sum_{\lambda  \in \F_p} I(\lambda) \left|\sum_{k =1}^K
  \chi\(\lambda +  k\)\right|
$$
where 
$$
 I(\lambda)  =\# \left\{(\ell,u,v)\in \cL  \times [1,U] \times [1,V] ~:~
\frac{v+r_u}{\ell} \equiv \lambda \pmod p\right\}.
$$

We now fix some integer $\nu \ge 1$. 
Writing 
$$I(\lambda)  =  I(\lambda) ^{(\nu-1)/\nu} ( I(\lambda) ^2)^{1/2\nu}
$$
and using the H{\"o}lder inequality, we derive
\begin{equation*}
\begin{split}
\fT ^{2\nu} & = \( \sum_{\lambda  \in \F_p}  I(\lambda) \)^{2\nu - 2}
\sum_{\lambda  \in \F_p}  I(\lambda) ^2 
 \sum_{\lambda  \in \F_p}  
 \left|\sum_{k =1}^K
\chi\(v + k\)\right|^{2\nu}. 
\end{split}
\end{equation*}
We obviously have 
\begin{equation}
\label{eq:I1}
\sum_{\lambda  \in \F_p}  I(\lambda)  \ll  \# \cL UV \ll LUV.
 \end{equation}
Hence, using Lemma~\ref{lem:N bound},  we obtain
\begin{equation}
\label{eq:I2}
\sum_{\lambda  \in \F_p}  I(\lambda) ^2 = N \le  LU^2V p^{-\kappa}
\end{equation}
for some $\kappa >0$ depending only on $\eta$ and thus only on $\varepsilon$.

Furthermore, taking $\nu = \rf{\gamma^{-1}}$, we derive
from  Lemma~\ref{lem:DavErd} that 
\begin{equation}
\label{eq:CharMoment nu}
\sum_{\lambda  \in \F_p}  
 \left|\sum_{k =1}^K
\chi\(v + k\)\right|^{2\nu} \ll K^{2\nu} p^{1/2} + K^\nu p \ll  K^{2\nu} p^{1/2}.
\end{equation}
Collecting the bounds~\eqref{eq:I1},  \eqref{eq:I2}
and~\eqref{eq:CharMoment nu}, we obtain
\begin{equation}
\label{eq:Sigma nu}
\begin{split}
\fT ^{2\nu} & \ll (LUV)^{2\nu-2}    LU^2V K^{2\nu}   p^{1/2-\kappa}\\
& =  (KLUV)^{2\nu}  (LV)^{-1}p^{1/2-\kappa}.
\end{split}
 \end{equation}
So if  $\delta \le \kappa/4$, we see that 
$$
 (LV)^{-1}p^{1/2-\kappa} = V^{-2} p^{1/2-3\kappa/5}\le p^{-\kappa/10}.
$$
Hence we infer from~\eqref{eq:Sigma nu} that 
$\fT  \ll  (KLUV) p^{-\kappa/20\nu}$. Substituting this inequality  
in~\eqref{eq:Sigma} and 
taking any
$$
\delta < \min\{\kappa/4, \kappa/20\nu, \gamma\}
= \kappa/20\nu, 
$$
we conclude  the proof.
\end{proof}

\section{Polynomial Congruences}

\subsection{Additive congruences with reciprocals}
\label{sec:cong recip}

We define $J_{d,k}(a,b;T)$ as the number of solutions to the 
congruence
$$
\sum_{j=1}^{2k}  \frac{(-1)^j}{(at_j+b)^d} \equiv 0 \pmod p, \qquad
1 \le t_1, \ldots, t_{2k} \le T.
$$
First we recall 
 the following result Bourgain \& Garaev~\cite[Proposition~1]{BouGar2}
(or~\cite[Theorem~1]{BouGar2} if $d=1$). 

\begin{lemma}
\label{lem:Jdk} For any fixed integer $k \ge 1$,  for $1 \le T< p$,  we have
$$
J_{d,k}(a,b;T) \le T^{2k+o(1)} p^{-1} + T^{2k^2/(k+1)+o(1)}. 
$$
\end{lemma}

For $d = 1$ we also write $J_{k}(a,b;T) = J_{1,k}(a,b;T)$. Then 
for $k=2$ and $d=1$, the result of  Cilleruelo \& Garaev~\cite[Theorem~1]{CillGar} (with $K=L$), see 
also~\cite[Corollary~10]{BouGar2}, yields

\begin{lemma}
\label{lem:J2} For $1 \le T< p$, we have
$$
J_2(a,b;T) \le T^{7/2+o(1)}p^{-1/2} +T^{2+o(1)}.
$$
\end{lemma}

\subsection{Multiplicative congruences with polynomials}
\label{sec:mult cong}

We need some results about the frequency of small values
amongst inverses modulo $p$ of polynomials. 
Namely, for a polynomial $f \in \F_p[X]$ and two positive integers 
$U$ and $Z$ we denote by $N_f(U,Z)$ the number of  solutions to 
the congruence 
\begin{equation}
\label{eq:cong f}
f(u) z \equiv 1 \pmod p, \qquad 1 \le u\le U, \ 1\le z \le Z. 
\end{equation}
Several approaches to estimating the number of solutions
of congruences of this types have recently been considered 
in~\cite{Chang,CCGHSZ,CillGar,CGOS,CSZ,Kerr}. 
For our purpose, the method of the proof of~\cite[Theorem~1]{CillGar}
is the most suitable.

We  now give one of our main technical results, which can be 
of independent interest. 

\begin{lemma}
\label{lem:NfUZ} Uniformly over  polynomials $f(X) \in \F_p[X]$ 
of degree $d\ge 1$, for $1 \le U,Z< p$,  we have 
$$
N_f(U,Z) \le  \( U^{d/2}Z p^{-1/(d+1)} +1\) p^{o(1)}.  
$$
\end{lemma}

\begin{proof} Without loss of generality we can assume that 
\begin{equation}
\label{eq:small U}
 U^{d/2} < p^{1/(d+1)}
\end{equation}
as otherwise that bound is weaker than 
the trivial bound $N_f(U,Z)  \ll Z$. 

Let us define positive $T_0, \ldots, T_d$ to satisfy
\begin{equation}
\label{eq:T_i}
T_0 = T_1U = \ldots = T_dU^d \mand T_0\ldots T_d = p^d.
\end{equation}
Clearly the conditions~\eqref{eq:T_i} allows us to find $T_i$
explicitly,  but for us only the common value   $W = T_i U^i$ is
important, $i =0, \ldots, d$.  
It is easy to see that $W^{d+1} = p^d U^{d(d+1)/2}$,
thus
$$
W = p^{d/(d+1)}  U^{d/2}.
$$
One verifies  that~\eqref{eq:small U} guarantees that 
$$
1 \le  WU^{-d} = T_d \le \ldots \le T_0 = W  < p.
$$
Thus  Lemma~\ref{lem:Reduction} applies with the above values
of $T_0, \ldots, T_d$. 

We now write $f(X) = a_dX^d +\ldots + a_1 X + a_0$ and find 
$t\in [1, p-1]$ as in Lemma~\ref{lem:Reduction}. 
We then define $g(X)=b_dX^d +\ldots + b_1 X + b_0$ 
with coefficients satisfying
$$b_i \equiv a_i t  \pmod p \mand |b_i|< p/2, 
\qquad i =0, \ldots, d.
$$

Then~\eqref{eq:cong f} is equivalent to 
\begin{equation}
\label{eq:cong g}
g(u) z \equiv t \pmod p, \qquad 1 \le u\le U, \ 1\le z \le Z, 
\end{equation}
and noticing that 
$$
\rho\(g(u) z\) \le (d+1)WZ
$$
for all $(u,z) \in [1,U]\times [1,Z]$, we see that~\eqref{eq:cong g}
implies  
$$
g(u) z  = t + sp
$$
for some $s = O(WZ)$.  In particular $z \mid t+sp$ and $t+sp \ne 0$.
We now recall the well-known bound
$$
\tau(m)  \le m^{o(1)} 
$$
on the number of integer positive divisors $\tau(m)$ 
of an integer $m\ne0$, 
see, for example,~\cite[Theorem~317]{HaWr}.

Hence, for each 
of the $O\(WZ/p + 1\) = O\(  U^{d/2}Z p^{-1/(d+1)} +1\)$ possible values of $k$,  there are 
at most $p^{o(1)}$ possible values for $z$, and for each $z$ 
at most $d$ values  for $u$. The 
result now follows.
\end{proof}

For $d =1$, that is, for a linear polynomial $f(X) = aX+b$,
the bound of  Lemma~\ref{lem:NfUZ} becomes  
$\(U^{1/2}Z p^{-1/2}  +1\) p^{o(1)}$ and  
corresponds to the estimate 
from~\cite{Shp2}.

\subsection{Additive congruences with polynomials}
\label{sec:add cong}

One of our  tools is also the following
very special case of a much more general bound
of Wooley~\cite{Wool}, that applies to polynomials with arbitrary real coefficients.

\begin{lemma}
\label{lem:Wooley}
 Uniformly over  polynomials $f(X) \in \F_p[X]$ 
of degree $d\ge 3$,  for $1 \le U< p$,  we have  
$$
\left|\sum_{u=1}^{U} \ep(f)\right| \ll U  \(U^{-1} + pU^{-d}\)^\sigma
$$
where 
$$
\sigma = \frac{1}{2(d-1)(d-2)}.
$$
\end{lemma}

Clearly, Lemma~\ref{lem:Wooley} is nontrivial for $U \ge p^{1/d + \varepsilon}$ 
for any fixed $\varepsilon >0 $ and sufficiently large $p$. 
In fact, the classical results of 
Vinogradov~\cite{Vin} are also sufficient for  our purposes.

Now combining 
Lemma~\ref{lem:Wooley} with the  {\it Erd\H{o}s-Tur\'{a}n inequality}
(see, for example,~\cite[Theorem~1.21]{DrTi}), that relates the uniformity of distribution to exponential sums, we immediately obtain:

\begin{lemma}
\label{lem:f in I} For any $\varepsilon > 0$ there exist some $\eta>0$ such that  
uniformly over  polynomials $f(X) \in \F_p[X]$ and
$b \in \F_p$ 
of degree $d\ge 3$,  for $p^{1/d + \varepsilon} \le U < p$ $1 \le U,Z< p$,  we have  
\begin{equation*}
\begin{split}
\# \{(u,z)  \in [1,U] \times [1,Z]~ :~ f(u)    \equiv b&+z \pmod p\}\\
& \qquad = \frac{UZ}{p} + O\(U p^{-\eta}\),
\end{split}
\end{equation*}
where $\Delta$ is as in Lemma~\ref{lem:Wooley}.
\end{lemma}

Obviously, the parameter $b$ does not add any generality to 
Lemma~\ref{lem:f in I} (compared to just $b =0$), 
however this is the form in which we apply it. 
In particular, we see that in the terminology of Section~\ref{sec:u-distr cong} the sequence 
of polynomial values $\{f(u)~:~u=1, \ldots, U\}$  is  $\eta$-uniformly distributed 
modulo $p$ provided that $U \ge p^{1/d + \varepsilon}$.

%

 \section{Proofs of Main Results} 

 \subsection{Proof of Theorem~\ref{thm:S pure}}

Since $\cC$ is convex,  as in the proof of Lemma~\ref{lem:CharSum bound} 
we see that for each $u=1, \ldots, U$ we there  are
 integers $V \ge Y_u \ge X_u \ge 0$ such that 
$$
S_f(\cA;\fC)=
\sum_{u =  1}^{U} \alpha_u  \sum_{v=X_u+1}^{Y_u}  \e_p(v/f(u)).
$$

We follow the scheme of the proof of~\cite[Lemma~3]{Shp1}.
In particular,  we define 
$$
I = \fl{\log(2p/V)} \mand J = \fl{\log (2p)}. 
$$

Furthermore, we extend the definition of $\rho(\alpha)$ from Section~\ref{sec:mult cong}
to rational numbers $\alpha = u/v$ with $\gcd(v,p)=1$,
 as  $\rho(\alpha) = |w|$  where $w$ is the unique integer with $w \equiv u/v \pmod p$
and $|w| < p/2$. 

Using the bound 
\begin{equation}
\label{eq:Incompl}
  \sum_{v=X_u+1}^{Y_u} \ep(\alpha v) 
\ll  \min\left\{V, \frac{p}{\rho(\alpha)}\right\},
\end{equation}
which holds for any rational $\alpha$ with the denominator that is not a 
multiple of $p$
(see~\cite[Bound~(8.6)]{IwKow}), we derive
\begin{equation}
\label{eq:Prelim}
S_f(\cA;\fC)=
\sum_{u =  1}^{U}  \min\left\{V, \frac{p}{\rho(a/f(u))}\right\}.
\end{equation}
Hence, 
we obtain a version 
of~\cite[Equation~(1)]{Shp1}: 
\begin{equation}
\label{eq:W RT}
S_f(\cA;\fC) \ll V R  
+ p \sum_{j = I+1}^J Q_{j} e^{-j},
\end{equation}
where
\begin{equation*}
\begin{split}
&R =\# \left\{u~:~ 1 \le u \le U,   \  
 \rho(a/f(u)) < e^{I} \right\},\\
&Q_{j} =\# \left\{u~:~ 1 \le u \le U, \  
e^j \le  \rho(a/f(u)) < e^{j+1} \right\}.
\end{split}
\end{equation*}

We now see that Lemma~\ref{lem:NfUZ} implies the 
bounds
$$
  R \le  \(p^{-1/(d+1)}  U^{d/2}e^I + 1\)  p^{o(1)} \le 
  p^{d/(d+1)}  U^{d/2} V^{-1}  + p^{o(1)}
$$
and
$$ 
  Q_j \le p^{-1/(d+1)+o(1)}  U^{d/2}e^{j} + p^{o(1)}.
$$
Substituting these bounds in~\eqref{eq:W RT}, we obtain 
\begin{equation*}
\begin{split}
 \left|S_f(\cA;\fC) \right|  &\ll  p^{d/(d+1)+o(1)}  U^{d/2}+ V p^{o(1)}\\
& \qquad\qquad\qquad \qquad+ p^{1+o(1)}\sum_{j = I+1}^J\(p^{-1/(d+1)}  U^{d/2}e^{j} +1\) e^{-j} \\
&   = \(p^{d/(d+1)}  U^{d/2}+ V 
+   J p^{d/(d+1) }  U^{d/2}  + p^{1 }e^{-I}\)p^{o(1)}\\
&   = p^{d/(d+1)+o(1)}  U^{d/2}+ V p^{o(1)},
\end{split}
\end{equation*}
which concludes the proof.

\subsection{Proof of Corollary~\ref{cor:S pure slice}}

Note that we can assume that $V >  p^{d/(d+1)}$
as otherwise the bound is trivial. 

We now choose some parameter $K\ge 1$ and slice the summation region $\fC$ vertically 
into $K$ domains of the form 
$\fC\cap [A+1, A+\widetilde{U}]$ with $\widetilde{U} = U/K + O(1)$. 
Applying Theorem~\ref{thm:S pure}
to the polynomial $f(A+X)$ to estimate the sum over $\fC\cap [A+1, A+\widetilde{U}]$,  we obtain 
$$
 \left|S_f(\cA;\fC) \right| \le K\( p^{d/(d+1)}( U/K)^{d/2}+ V\) p^{o(1)}. 
$$
We now set 
$$K = \rf{UV^{-2/d} p^{2/(d+1)}} 
$$
and conclude the proof.

\subsection{Proof of Theorem~\ref{thm:Sabd mix}}

Since $\fC$ is convex, for each $v \in [1,V]$ 
there are some integers $U \ge Y_v\ge X_v\ge 0$ such that for 
an integer $u$ the condition
$(u,v) \in \fC$ is equivalent to $X_u < u \le Y_u$. 
Hence, using the orthogonality of exponential functions, we obtain 
\begin{equation*}
\begin{split}
\sum_{u:\, (u,v) \in \fC} &\alpha_u  \ep\(v/(au+b)^d\)   =
\sum_{u= X_v+1}^{Y_v} \alpha_u \ep\(v/(au+b)^d\)\\
& = \sum_{u=1}^U\alpha_u \ep(v/(au+b)^d)\sum_{x= X_v+1}^{Y_v}\frac{1}{p}
\sum_{\lambda=-(p-1)/2}^{(p-1)/2}
 \ep(\lambda(x-u))\\
& = 
\frac{1}{p} \sum_{\lambda=-(p-1)/2}^{(p-1)/2}
\sum_{u=1}^U\alpha_u  \ep(v/(au+b)^d - \lambda u)
\sum_{x= X_v+1}^{Y_v} \ep(\lambda x).
\end{split}
\end{equation*}
Recalling~\eqref{eq:Incompl}, we obtain 
\begin{equation*}
\begin{split}
 \left|\sum_{u:\, (u,v) \in \fC}  \alpha_u \ep\(v/(au+b)^d\)\right| &\\
 \ll  
\sum_{\lambda=-(p-1)/2}^{(p-1)/2} \frac{1}{|\lambda|+1} &
 \left| \sum_{u=1}^U\alpha_u  \ep(v/(au+b)^d - \lambda u)\right|. 
 \end{split}
\end{equation*}
Thus, we  conclude that 
\begin{equation}
\label{eq:K absval}
\begin{split}
|K_{a,b}(\cA, \cB; \fC)| & \le 
\sum_{v =  1}^{V} \left|\sum_{u:\, (u,v) \in \fC}  \alpha_u  \ep\(v/(au+b)^d\) \right|\\
& \ll \sum_{\lambda=-(p-1)/2}^{(p-1)/2} \frac{1}{|\lambda|+1} W(\lambda),
 \end{split}
\end{equation}
where 
$$
W(\lambda)=
\sum_{v =  1}^{V}
 \left| \sum_{u=1}^U\alpha_u(\lambda)  \ep\(v/(au+b)^d\)\right|
$$
with $\alpha_u(\lambda) = \alpha_u \ep (- \lambda u)$. 
Using the H{\"o}lder inequality and the extending the range of 
summation over $v$ to teh whole field $\F_p$, 
we derive
\begin{equation*}
\begin{split}
W(\lambda)^{2k} &\le V^{2k-1}
\sum_{v =  1}^{V}
 \left| \sum_{u=1}^U\alpha_u(\lambda)  \ep\(v/(au+b)^d\)\right|^{2k}\\
  &\le V^{2k-1}
\sum_{v\in \F_p}
 \left| \sum_{u=1}^U\alpha_u(\lambda)  \ep\(v/(au+b)^d\)\right|^{2k}. 
\end{split}
\end{equation*}
Using that for any complex $\xi$ we have $|\xi|^2 = \xi \overline \xi$,
expanding the $(2k)$-th power and changing the order of summation, 
we derive
\begin{equation*}
\begin{split}
W(\lambda)^{2k} &\le V^{2k-1}
\sum_{u_1, \ldots, u_{2k}=1}^U
\prod_{i=1}^{k} \alpha_{u_{2i-1}}(\lambda) 
\overline{  \alpha_{u_{2i}}(\lambda)}\\
& \qquad \qquad \qquad \qquad \qquad \qquad \sum_{v\in \F_p}
\ep\(v \sum_{j=1}^{2k}  \frac{(-1)^{j}}{(au_j+b)^d}\)\\
&\le V^{2k-1}
\sum_{u_1, \ldots, u_{2k}=1}^U
\sum_{v\in \F_p}
\ep\(v \sum_{j=1}^{2k}  \frac{(-1)^{j}}{(au_j+b)^d}\) \\
&=  V^{2k-1} p J_{d,k}(a,b;U), 
\end{split}
\end{equation*}
where $J_{d,k}(a,b;U)$ is as in Section~\ref{sec:cong recip}.
From Lemma~\ref{lem:Jdk}  we now easily derive
\begin{equation}
\label{eq:W bound}
W(\lambda) \le V^{1-1/2k}  \( U + U^{k/(k+1)} p^{1/2k}\) p^{o(1)},
\end{equation}
which we  substitute  in~\eqref{eq:K absval} and 
conclude the proof.

\subsection{Proof of Theorem~\ref{thm:K mix 1}}  
We proceed exactly as in the proof of Theorem~\ref{thm:Sabd mix} with $k=2$
but use   Lemma~\ref{lem:J2} instead of Lemma~\ref{lem:Jdk},
getting 
$$
W(\lambda)^4 \le V^{3} p\(U^{7/2+o(1)}p^{-1/2} +U^{2+o(1)}\)
$$
instead of~\eqref{eq:W bound}.

\subsection{Proof of Theorem~\ref{thm:K mix 2}} 

Writing 
$$
|K_{a,b}(\cA, \cB; \fC)| \le 
\sum_{v =  1}^{V} \left|\sum_{u:\, (u,v) \in \fC} \alpha_u \ep(v/(au+b))\right|
$$
and using the Cauchy inequality, we obtain 
\begin{equation*}
\begin{split}
|K_{a,b}(\cA&, \cB; \fC)|^2 \le 
V \sum_{v =  1}^{V} \left|\sum_{u:\, (u,v) \in \fC} \alpha_u \ep(v/(au+b))\right|^2\\
&= V \sum_{v =  1}^{V}\sum_{u:\, (u,v) \in \fC}\, \sum_{w:\, (w,v) \in \fC}  
\alpha_{u} \alpha_{w} \ep\(\frac{v}{au+b}-\frac{v}{aw+b}\)\\
&= V \sum_{v =  1}^{V}\sum_{u:\, (u,v) \in \fC} \, \sum_{w:\, (w,v)  \in \fC}  
\alpha_{u} \alpha_{w} \ep\(\frac{av(w-u)}{(au+b)(aw+b)}\).
\end{split}
\end{equation*}
Hence, changing the order of summation again, we obtain
$$
|K_{a,b}(\cA, \cB; \fC)|^2 \le 
V  \sum_{u,w=1}^U   
\left|\sum_{v\in \cI(u,w)}\ep\(\frac{av(w-u)}{(au+b)(aw+b)}\)\right|,
$$
where $\cI(u,w)$ is the set of integers $v \in [1,V]$ 
with $(u,v) \in \fC$ and $(w,v) \in \fC$. Therefore,  writing $w = u+z$, 
we derive
\begin{equation}
\label{eq:K uwz}
\begin{split}
|K_{a,b}(\cA&, \cB; \fC)|^2 \\
& \le 
V \sum_{z=-U}^U   \sum_{u=1}^U   
\left|\sum_{v\in \cI(u,u+z)}\ep\(\frac{avz }{(au+b)(a(u+z)+b)}\)\right|. 
\end{split}
\end{equation}

Now, for $z=0$ we estimate the double   sum over $u$ and $v$ 
trivially as $UV$.
Otherwise, we note that $\cI(u,u+z)$ is an interval and
then  proceed exactly as in the proof of  Theorem~\ref{thm:S pure}
arriving to ~\eqref{eq:Prelim}. Hence, for every $z\ne 0$, 
to the sum over $u$ and $v$ in~\eqref{eq:K uwz} we obtain 
the bound of  Theorem~\ref{thm:S pure} with $d =2$. Collecting 
these estimates, we obtain  
\begin{equation*}
\begin{split}
|K_{a,b}(\cA, \cB; \fC)|^2 & \le UV^2 + UV\(p^{2/3+o(1)}  U+ V p^{o(1)}\)\\
& = U^2V p^{2/3+o(1)}+  UV^2 p^{o(1)}, 
\end{split}
\end{equation*}
and the result now follows.

\subsection{Proof of Theorem~\ref{thm:T pure}}

The result follows immediately from the combinations of Lemmas~\ref{lem:CharSum bound} 
and~\ref{lem:f in I}.

\section{Comments}
\label{sec:comm}

 Theorems~\ref{thm:S pure} and~\ref{thm:T pure} also
improve the generic bounds~\eqref{eq:trivial S1}, 
\eqref{eq:trivial S2} 
and~\eqref{eq:trivial T} for a very wide range of parameters. 

Our arguments also apply to the sums
$$
\sum_{(u,v)\in \fC} \alpha_u \beta_v \e_p(v f(u)),\qquad f\in \F_p[X],
$$
where one can use the results from~\cite{CCGHSZ,CGOS}
instead of  Lemma~\ref{lem:NfUZ}. However, after an application 
of the Cauchy inequality and ``smoothing'' the summation over $u$ one 
can also use the bounds of Wooley~\cite{Wool} directly, while our 
approach does not seem to give any substantial gain. 

We also remark that~\eqref{eq:Prelim} immediately implies the bound 
 \begin{equation}
\label{eq:trivial S2}
S_f(\cA; \fC) \ll p \log p, 
 \end{equation}
  which is better than~\eqref{eq:trivial S1} (however~\eqref{eq:trivial S1} 
  applies to more general sums). 
 
It is easy to see that for $d=1$ each of Theorems~\ref{thm:Sabd mix}, 
 \ref{thm:K mix 1} and~\ref{thm:K mix 2} has a range of parameters
 where it is stronger than the other two (and the
  bound~\eqref{eq:trivial S2}).

 It easy to see that for $d=1$ each of the Theorems~\ref{thm:Sabd mix}, 
 \ref{thm:K mix 1} and~\ref{thm:K mix 2} has a range of parameters
 where it is stronger than the other two (and the
 generic bound~\eqref{eq:trivial S2}.

 In Table~\ref{tab:compare} we give examples which illustrate
 this (`$\ast$' indicates the winning bound and `---' indicates that 
 the corresponding bound is trivial. Note that in all out examples
 $UV \le p$, the bound ~\eqref{eq:trivial S2} is always trivial and
 so is not included in  Table~\ref{tab:compare}.
We also suppress  
 the terms $p^{o(1)}$). In the case of Theorem~\ref{thm:Sabd mix}, 
we also give the optimal values of $k$. 
\begin{table}[H]
  \centering
\begin{tabular}{|c|c|c|c|c|}
\hline
$(U,V)$   & Thm.~\ref{thm:Sabd mix}  & 
Thm.~\ref{thm:K mix 1} & Thm.~\ref{thm:K mix 2} \\
\hline
$(p^{2/5},p^{3/10})$   & $\ast~p^{419/600}~(k=14\text{ or } 15)$&---&---\\ 
$(p^{2/5},p^{2/5})$  & $p^{111/140}~(k=6 \text{ or } 7)$&$\ast~p^{31/40}$&---\\
$(p^{1/5},p^{4/5})$  & $p^{59/60}~(k=2 \text{ or } 3)$&$ p^{19/20}$&$\ast~p^{14/15}$\\
\hline
\end{tabular}\medskip
\caption{Comparison between Theorems~\ref{thm:Sabd mix}, 
 \ref{thm:K mix 1} and~\ref{thm:K mix 2}}
\label{tab:compare}
\end{table}

It is certainly interesting to obtain nontrivial estimates
on  the sums $S_f(\cA, \cB; \fC)$
and $T_f(\cA, \cB; \fC)$
in full generality or even on more general sums, where $f(X)$ is a 
rational function rather than a polynomial.
For example, for  the sums $S_f(\cA, \cB; \fC)$, to get such result one
needs appropriate extensions of the bounds 
on $N_f(U,Z)$ in Section~\ref{sec:mult cong} to arbitrary rational 
functions. Obtaining such bounds is certainly of independent 
interest. In fact it is easy to see that
a standard application of the Weil bound of exponential sums with rational
functions, see, for example,~\cite{MorMor}, leads to the asymptotic 
formula
\begin{equation}
\label{eq:Weil}
N_f(U,Z) = \frac{UZ}{p} + O(p^{1/2} (\log p)^2), 
\end{equation}
for any nontrivial rational function $f$. It is obvious that~\eqref{eq:Weil} 
is nontrivial for $U\ge p^{1/2 + \varepsilon}$, with  any fixed 
$\varepsilon$ and thus can be used to estimate the sum  $S_f(\cA; \fC)$
and $S_f(\cA, \cB; \fC)$. Furthermore, in this range one can simply use the 
Cauchy inequality to obtain a full analogue of~\eqref{eq:K absval}
for the sums $S_f(\cA, \cB; \fC)$, and then apply to the inner sums  the Weil bound
for incomplete sums with rational functions. The same standard approach also works
for the sums $T_f(\cA, \cB; \fC)$. However we are mostly interested in 
small values of $U$, beyond  the reach of the Weil bound.

%

\end{document}